\documentclass[10pt,a4paper]{amsart}
\usepackage[english]{babel}
\usepackage{amsmath,amsthm,mathrsfs,hyperref}
\usepackage{amssymb}
\usepackage{aliascnt}
\usepackage{xspace}
\usepackage{tikz-cd} 
\usepackage{comment}

\DeclareMathOperator{\dom}{dom}
\DeclareMathOperator{\range}{range}

\DeclareMathOperator{\crit}{crit}

\DeclareMathOperator{\cf}{cf}

\DeclareMathOperator{\Add}{Add}
\DeclareMathOperator{\Ult}{Ult}
\DeclareMathOperator{\Hull}{Hull}

\newcommand{\ZFC}{{\rm ZFC}\xspace}
\newcommand{\ZF}{{\rm ZF}\xspace}

\newcommand{\Ord}{{\rm Ord}\xspace}

\usepackage[textsize=footnotesize,color=yellow!40, bordercolor=white]{todonotes}

\newcommand{\AD}{\ensuremath{\operatorname{AD}} }
\newcommand{\DC}{\ensuremath{\operatorname{DC}} }

\newcommand{\Col}{\ensuremath{\operatorname{Col}} }

\newcommand{\trcl}{\ensuremath{\operatorname{trcl}} }

\newcommand{\sta}{*}

\DeclareMathOperator{\len}{lh}
\newcommand{\lh}{\len}

\newcommand{\cP}{\mathcal{P}}
\newcommand{\Pot}{\mathcal{P}}
\newcommand{\cM}{\mathcal{M}}
\newcommand{\cN}{\mathcal{N}}

\newcommand{\cR}{\mathcal{R}}

\newcommand{\bR}{\mathbb{R}}
\newcommand{\bP}{\mathbb{P}}

\newtheorem{theorem}{Theorem}
\newaliascnt{example}{theorem}

\aliascntresetthe{example}
\newaliascnt{fact}{theorem}

\aliascntresetthe{fact}
\newaliascnt{corollary}{theorem}
\newtheorem{corollary}[corollary]{Corollary}
\aliascntresetthe{corollary}
\newaliascnt{lemma}{theorem}
\newtheorem{lemma}[lemma]{Lemma}
\aliascntresetthe{lemma}
\newaliascnt{claim}{theorem}
\newcounter{cl}[theorem]
\newtheorem{claim}[cl]{Claim}
\aliascntresetthe{claim}
\newtheorem{prop}[theorem]{Proposition}
\theoremstyle{definition}
\newaliascnt{definition}{theorem}
\newtheorem{definition}[definition]{Definition}
\aliascntresetthe{definition}
\newtheorem{question}{Question}

\newtheorem*{theorem*}{Theorem}

\newtheorem*{example*}{Example}

\title{Perfect Subtree Property for weakly compact cardinals}

\author{Yair Hayut} 
\address[Y. Hayut]{Einstein Institute of Mathematics, The Hebrew University of Jerusalem. Jerusalem 91904, Israel}
\email{yair.hayut@mail.huji.ac.il}
\thanks{The first-listed author was supported by Austrian Science Fund (FWF) Lise Meitner grant 2650-N35}

\author{Sandra M\"uller}
\address[S. M\"uller]{Institut f\"ur Diskrete Mathematik und Geometrie, TU Wien, Wiedner Hauptstra{\ss}e 8-10/104, 1040 Wien, Austria, \and Institut f\"ur Mathematik, Universit\"at Wien, Kolingasse 14-16, 1090 Wien, Austria}
\email{sandra.mueller@tuwien.ac.at}
\urladdr{http://www.logic.univie.ac.at/~smueller/}
\thanks{The second-listed author was supported by L'OR\'{E}AL Austria, in collaboration with the Austrian UNESCO Commission and in cooperation with the Austrian Academy of Sciences - Fellowship \emph{Determinacy and Large Cardinals} and Elise Richter grant number V844 of the FWF}

\date{\today}

\begin{document}
\begin{abstract}
We investigate the consistency strength of the statement: $\kappa$ is weakly compact and there is no tree on $\kappa$ with exactly $\kappa^{+}$ many branches. We show that this statement fails strongly (in the sense that there is a \emph{sealed tree} with exactly $\kappa^{+}$ many branches) if there is no inner model with a Woodin cardinal. Moreover, we show that for a weakly compact cardinal $\kappa$ the nonexistence of a tree on $\kappa$ with exactly $\kappa^{+}$ many branches and, in particular, the Perfect Subtree Property for $\kappa$, implies the consistency of $\AD_{\bR} + \DC$.  
\end{abstract}
\maketitle
\section{Introduction}
Trees and their collections of branches play an important role in topology and infinitary combinatorics. Indeed, closed subsets of the Cantor Set $2^\omega$  are exactly the collections of branches of subtrees of $2^{<\omega}$. By the Cantor-Bendixson theorem, subtrees of the binary tree on $\omega$ satisfy a dichotomy - either the tree has countably many branches or there is a perfect subtree (and in particular, the number of branches of the tree is the continuum, regardless of the size of the continuum). Equivalently, if a tree $T\subseteq 2^{<\omega}$ has uncountably many branches and $\mathbb{P}$ is a forcing notion that adds a new real then $\mathbb{P}$ adds a new branch through $T$. 

\begin{definition}
Let $\kappa$ be a regular cardinal. The \emph{Branch Spectrum} of $\kappa$ is the set \[\mathfrak{S}_\kappa = \{|[T]| \mid T \text { is a normal }\kappa\text{-tree}\}.\]
\end{definition}
For the definition of normal trees see the beginning of Section \ref{section: upper bounds}. A first example is $\mathfrak{S}_\omega = \{2^{\aleph_0}\}$. The spectrum $\mathfrak{S}_{\omega_1}$ was first studied by Silver \cite{Si71}, showing the independence of the Kurepa Hypothesis. Later, questions about the possible values of this spectrum were addressed by Shelah and Jin in \cite{ShJi92} and more recently a complete characterization was given by Po{\'o}r and Shelah \cite{Po, PoSh20}. By \cite{SiSo}, the branch spectrum is related to the model theoretic spectrum of maximal models of $\mathcal{L}_{\omega_1,\omega}$-sentences.  

For uncountable cardinals $\kappa$, the assertion $\max(\mathfrak{S}_\kappa) = \kappa$ (i.e., there are no $\kappa$-Kurepa trees) has the consistency strength of an inaccessible cardinal. The assertion $\min(\mathfrak{S}_\kappa) = \kappa$ (i.e., the tree property at $\kappa$ for $\kappa$ with uncountable cofinality) has the consistency strength of a weakly compact cardinal. We are interested in the consistency strength of the conjunction of these two statements, i.e., the non-existence of $\kappa$-Kurepa trees and the tree property. Since we are interested in inaccessible cardinals, we replace the first statement by the weaker assertion $\kappa^+ \notin \mathfrak{S}_\kappa$. As usual for these types of properties, where individually each one of them has a mild consistency strength, their combination is very strong. In Section \ref{section: applications} we will show that the statement $\kappa^{+}\notin\mathfrak{S}_\kappa$ for a weakly compact cardinal $\kappa$ implies the consistency of $\AD_{\bR} + \DC$.  

Trees with (somewhat) absolute sets of branches play a role in the derivation of consistency strength from certain configurations of the branch spectrum in the context of some covering lemma.
Indeed, if $\kappa$ is a regular cardinal and $L$ computes $\kappa^{+}$ correctly then there is a tree $T\subseteq 2^{<\kappa}$ with exactly $\kappa^{+}$ many branches, and moreover any model in which $\kappa$ has uncountable cofinality does not have any non-constructible cofinal branch through $T$. This result, which is a variant of a construction due to Solovay, generalizes to the Jensen-Steel core model $K$ below inner models with a Woodin cardinal, replacing absoluteness by set forcing absoluteness. 

\begin{definition}\label{def:strongly sealed}
Let $\kappa$ be a regular cardinal. A normal tree $T$ of height $\kappa$ is \emph{strongly sealed} if the set of branches of $T$ cannot be modified by set forcing that forces $\cf(\kappa) > \omega$.
\end{definition}
Strongly sealed trees with $\kappa$ many branches exist in $\ZFC$: take $T \subseteq 2^{<\kappa}$ to be the tree of all $x$ such that $\{\alpha \in \dom(x) \mid x(\alpha) = 1\}$ is finite. Thus, our main interest is in strongly sealed $\kappa$-trees with at least $\kappa^+$ many branches. Our constructions are very inner model theoretical, and thus can only produce $\kappa$-trees with $\kappa^{+}$ many branches, where $\kappa^{+}$ is computed correctly in some inner model. 
\begin{question}\label{question:stronglysealedkappa++tree}
Is it consistent that there is a strongly sealed $\kappa$-tree with $\kappa^{++}$ many branches?
\end{question}

Note that strongly sealed $\kappa$-trees do not exist if there is a Woodin cardinal $\delta > \kappa$, since Woodin's stationary tower forcing with critical point $\kappa^+$ will introduce new branches through any $\kappa$-tree $T$, while preserving the regularity of $\kappa$, as well as many large cardinal properties of $\kappa$. Thus, in order to get a meaningful notion of sealed trees in the presence of large cardinals we will use a weaker notion. The weakest notion of sealed tree is arguably having no perfect subtree (a perfect subtree is a continuous copy of $2^{<\kappa}$).
\begin{lemma}[Folklore]\label{lem:folkloreequiv}
Let $\kappa$ be a cardinal. The following are equivalent for a tree $T$ of height $\kappa$:
\begin{enumerate}
\item $T$ has a perfect subtree. 
\item Every set forcing that adds a fresh subset to $\kappa$ also adds a branch through $T$. 
\item There is a $\kappa$-closed forcing that adds a branch through $T$.
\end{enumerate}
\end{lemma}
See Lemma \ref{lemma: embedding cantor trees} for the argument for $(3)\implies (1)$.

\begin{definition}\label{def:psp}
Let $\kappa$ be an uncountable cardinal. The \emph{Perfect Subtree Property} (PSP) for $\kappa$ is the statement that every $\kappa$-tree with at least $\kappa^+$ many branches has a perfect subtree.
\end{definition}

The Perfect Subtree Property can easily be violated by small forcing by Hamkins' Gap Forcing argument, \cite{Ha01}, see also Proposition \ref{prop:sigmaclosedcards}. In addition, as we will see in Section \ref{section: applications}, in the presence of some covering lemma there is a natural counterexample to the PSP in an inner model, providing a lower bound for the consistency of the PSP as well as a natural intermediate notion of \emph{sealed tree} which is compatible with the existence of Woodin cardinals. In particular, Theorem \ref{theorem: psp and AD_R} shows that the existence of a sealed $\kappa$-tree is consistent with the existence of Woodin cardinals above $\kappa$.

\begin{definition}\label{def:sealed}
  Let $\kappa$ be a regular cardinal. A normal tree $T$ of height $\kappa$ is \emph{sealed} if the set of branches through $T$ cannot be modified by set forcing $\bP$ satisfying the following properties:
  \begin{enumerate}
  \item $\bP \times \bP$ does not collapse $\kappa$,
  \item $\bP \times \bP$ preserves $\cf(\kappa) > \omega$, and
  \item $\bP$ does not add any new sets of reals.
  \end{enumerate}
\end{definition}

Note that Woodin's stationary tower forcing with arbitrary critical point does not satisfy these properties. The class of forcings $\bP$ is designed to preserve inner model theoretic properties such as iterability of mice and some form of condensation. This class of forcings contains $\kappa$-closed forcings such as $\Add(\kappa,1)$ and $\Col(\kappa, \lambda)$. In particular, if a tree $T$ is sealed then it has no perfect subtree. This notion of being sealed leads to a variant of Question \ref{question:stronglysealedkappa++tree} that is also open.

The structure of the paper is as follows. In Section \ref{section: upper bounds} we will give a few forcing constructions providing an upper bound for the consistency of PSP at a weakly compact cardinal as well as $\kappa^+ \notin \mathfrak{S}_\kappa$. These forcing constructions are mostly folklore. 
%In Section \ref{section: comparison} we review some inner model theoretic notions and show a technical lemma which we will use for our main results. 
In Section \ref{section: abstract construction} we will revisit Solovay's theorem on the existence of a Kurepa tree in $L$ from \cite[Section 4]{Je71}, which is the main ingredient in his proof for the consistency strength of the Kurepa Hypothesis. 

As noted by Po{\'o}r and Shelah \cite{PoSh20}, a variant of Solovay's construction from \cite[Section 4]{Je71} provides a strongly sealed Kurepa tree. We conclude that if $0^\#$ does not exist then every weakly compact cardinal carries a strongly sealed tree with $\kappa^+$ many branches. The argument extends to the Jensen-Steel core model $K$ below proper class inner models with a Woodin cardinal:
\begin{theorem}\label{theorem: sealed trees in K}
Assume that there is no proper class inner model with a Woodin cardinal. 
Then for every regular cardinal $\kappa \geq \aleph_2$, there is a strongly sealed $\kappa$-tree with exactly $\left(\kappa^{+}\right)^K$ many branches. In particular, if $\kappa$ is weakly compact, then there is a strongly sealed $\kappa$-tree with $(\kappa^{+})^V$ many branches.
\end{theorem}
The abstract construction of such a tree is done in Section \ref{section: abstract construction}. Note that with a more elaborate fine structural argument our construction is likely to also yield a strongly sealed $\kappa$-tree with exactly $\left(\kappa^{+}\right)^K$ many branches in the case $\kappa = \omega_1$ which is not covered by our proof of Theorem \ref{theorem: sealed trees in K} but this is beyond the scope of this paper.
In Section \ref{section: applications} we use the abstract construction to prove Theorem \ref{theorem: sealed trees in K} as well as the following result:
\begin{theorem}\label{theorem: psp and AD_R}
Let $\kappa$ be a weakly compact cardinal and assume that there is no non-domestic premouse\footnote{See Definition \ref{def:domesticmouse} below.}. Then there is a sealed $\kappa$-tree with exactly $\kappa^+$ many branches.
\end{theorem}

By Lemma \ref{lem:folkloreequiv}, this yields the following corollary.

\begin{corollary}
  Let $\kappa$ be a weakly compact cardinal and suppose the Perfect Subtree Property holds at $\kappa$. Then there is a non-domestic premouse. In particular, there is an inner model of $\ZF + \DC + \AD_{\bR}$. 
\end{corollary}

Our definitions are mostly standard. For facts about forcing and trees we refer the reader to \cite[Chapters 9, 14]{Je03}. For basic facts, definitions and notions related to inner model theory, we refer the reader to \cite{St10} with the exception that we are using \emph{Jensen indexing} as in \cite{Je}. This is for example used in \cite{JSSS09}. We assume throughout the paper that there is no mouse with a superstrong extender.

We would like to thank the anonymous referee for a long list of corrections and suggestions that greatly improved the article.

\section{Trees and upper bounds}\label{section: upper bounds}
We will use the following definition of Kurepa tree. 
\begin{definition}
Let $\kappa$ be a cardinal and let $T$ be a tree of height $\kappa$. Write $[T]$ for the set of all cofinal branches through $T$. Then $T$ is called \emph{$\kappa$-Kurepa} if $T$ is a $\kappa$-tree and $|[T]| \geq \kappa^{+}$. 
\end{definition}
Note that some authors require a Kurepa tree to have exactly $\kappa^{+}$ branches.

We say that a tree of height $\kappa$ is \emph{binary} if it is a subtree of $2^{<\kappa}$, the \emph{full binary tree}.  We say that a tree is \emph{normal} if every node splits, every node has an arbitrarily high node above it, and at every limit level, each nodes is uniquely determined from the branch below it. 
When $\kappa$ is strongly inaccessible, the full binary tree is a normal Kurepa tree. Thus, in those cases the notion of \emph{slim Kurepa tree} is more suitable but we will not pursue this direction here. We remark that if $T$ is a $\kappa$-tree then there is a normal tree $S \subseteq T$ (the pruned subtree of $T$) such that $|[S]| + \kappa \geq |[T]|$. In particular, focusing on trees with at most $\kappa$ many branches, $S$ and $T$ have the same number of branches. 

In \cite{Si71}, Silver showed that if $\mu < \kappa$ are cardinals where $\mu$ is a successor cardinal and $\kappa$ is inaccessible, then there are no $\mu$-Kurepa trees in the generic extension by the Levy collapse $\Col(\mu, {<}\kappa)$. We review the argument here, as it is one of the motivations for Definition \ref{def:psp}.

Let $\dot{T}$ be a name for a $\mu$-Kurepa tree in the generic extension. Using the chain condition of the Levy collapse, $\dot{T}$ can be represented as a name with respect to some initial segment of the collapse, $\Col(\mu, {<}\bar\kappa)$, for $\bar\kappa < \kappa$. In the intermediate model, $V^{\Col(\mu, {<}\bar\kappa)}$, $2^\mu \leq (2^{\bar{\kappa}})^{V} < \kappa$, and thus $\dot{T}$ has at most $2^{\bar{\kappa}}$ many branches. Since this cardinal is collapsed in the full generic extension, in order for $\dot{T}$ to have more than $\mu$ many branches, the quotient forcing $\Col(\mu, [\bar{\kappa}, \kappa) )$ must introduce them. The following lemma is the crucial step in the proof:
\begin{lemma}\label{lemma: embedding cantor trees}
Let $T$ be a tree of height $\mu$ and let $\mathbb{P}$ be a $\mu$-closed forcing. If $\mathbb{P}$ introduces a new branch through $T$ then there is an order preserving injection from the full binary tree $2^{<\mu}$ to $T$.
\end{lemma}
\begin{proof}
Let $\dot{b}$ be a name for the new branch. Let us construct by induction for every $\eta \in 2^{<\mu}$ a pair $(p_\eta, t_\eta)$ where $p_\eta \in \mathbb{P}$ is a condition, $t_\eta \in T$ and $p_\eta \Vdash \check{t}_\eta \in \dot{b}$. We will assume, by induction, that for every $\eta \trianglelefteq \eta'$, $p_{\eta'}$ is stronger than $p_\eta$, $t_{\eta'}$ is above $t_\eta$ in the tree and $t_{\eta \frown \langle 0\rangle}, t_{\eta \frown \langle 1\rangle}$ are on the same level $\xi_\eta$ and incompatible.

At limit steps, we use the closure of the forcing to define $p_\eta$ as a condition stronger than $p_{\eta \restriction \alpha}$, for all $\alpha < \len(\eta)$. Since $p_\eta$ is a condition in $\mathbb{P}$ and it forces $t_{\eta \restriction \alpha} \in \dot{b}$ for all $\alpha < \len(\eta)$, there is an element of $T$ above all $t_{\eta \restriction \alpha}$ and $p_\eta$ forces the $\leq_T$-least such element to be in $\dot{b}$. 

At successor steps, we use the assumption that $\dot{b}$ is new in order to find two different extensions of $p_\eta$ that force incompatible values $t_{\eta \frown \langle 0 \rangle}, t_{\eta \frown \langle 1 \rangle}$ for the element of the branch $\dot b$ at some level $\xi_\eta$ above $\len(t_\eta)$. 
\end{proof}

Notice that for a successor cardinal $\mu$, there is no such embedding of the full binary tree into a $\mu$-tree. Indeed, let $\rho$ be the least cardinal such that $2^\rho \geq \mu$ (necessarily, $\rho < \mu$ since $\mu$ is a successor cardinal). Let us consider the $\rho$-th level of the binary tree. By our assumption, it consists of $2^\rho \geq \mu$ many different elements. But since $2^{<\rho} < \mu$, the levels of the images of the elements of the binary tree below $\rho$ are bounded and by continuity so are their limits, which is a contradiction.  

The proof shows that for every regular $\mu$, in a generic extension by the forcing  $\Col(\mu, {<}\kappa)$, every $\mu$-tree which has $\mu^+$ many branches contains a perfect subtree. By essentially the same argument, one can show that if $\mu$ is strongly inaccessible then after forcing with $\Col(\mu, {<}\kappa) \times \Add(\mu, \kappa^+)$, every $\mu$-tree either has ${\leq}\mu$ many branches or a perfect subtree, and therefore $\mu^{++}$ many branches. So, in this model $\mathfrak{S}_\mu \subseteq \mu \cup \{\mu, 2^\mu\}$ and $2^\mu > \mu^+$. If we further assume that $\mu$ is weakly compact in the generic extension, then $\mathfrak{S}_\mu = \{\mu, 2^{\mu}\}$. 

The only known way to preserve weak compactness after the Levy collapse is to start with a supercompact cardinal and force with some preparation forcing. There is some evidence that this is necessary, see \cite{ApHa01, NeSt16}.  
Altogether, we showed the following:
\begin{prop}
Let $\mu$ be ${<}\kappa$-supercompact, where $\kappa$ is strongly inaccessible. Then there is a forcing extension in which $\mu$ is weakly compact, $\mathfrak{S}_\mu = \{\mu, \mu^{++}\}$ and the Perfect Subtree Property holds at $\mu$.  
\end{prop}

The next proposition shows that the Perfect Subtree Property does not follow from $\mu^{+}\notin \mathfrak{S}_\mu$. 

We say that a cardinal $\mu$ is \emph{$\sigma$-closed} if $\rho^{\aleph_0}< \mu$ for every $\rho < \mu$.

\begin{prop}\label{prop:sigmaclosedcards}
Let $\mu$ be a regular cardinal. Adding a single Cohen real forces that the Perfect Subtree Property fails at $\mu$. If we further assume that $\mu$ is $\sigma$-closed and for all $\lambda \in \mathfrak{S}_\mu \setminus \mu$, $\lambda^{\aleph_0} = \lambda$ then $\mathfrak{S}_\mu^V = \mathfrak{S}_\mu^{V[G]}$, where $G$ is generic for adding a single Cohen real. 
\end{prop}
This immediately gives the following corollary.
\begin{corollary}
Let $\mu < \kappa$ be cardinals such that $\kappa$ is inaccessible and $\mu$ is ${<}\kappa$-supercompact. Then there is a generic extension in which $\mathfrak{S}_\mu = \{\mu, \mu^{++}\}$ but the Perfect Subtree Property fails at $\mu$.
\end{corollary}
\begin{proof}[Proof of Proposition \ref{prop:sigmaclosedcards}]
The first assertion follows from \cite{Ha01}. 

Let us address the preservation of the spectrum. A Cohen real cannot add a new cofinal branch through an old tree of regular height $\geq \omega_1$, as it is productively c.c.c., (see, for example, \cite{Unger15}). Thus, $\mathfrak{S}_{\mu}^{V[G]} \supseteq \mathfrak{S}_\mu^V$. We need to show that the other inclusion holds as well.

We need to show that if there is a $\mu$-Aronszajn tree in $V[G]$, and therefore $0\in \mathfrak{S}_\mu^{V[G]}$, then there is a $\mu$-Aronszajn tree in $V$. It follows that in this case $\mathfrak{S}_\mu^V \cap \mu = \mathfrak{S}_\mu^{V[G]}\cap \mu = \mathrm{Card} \cap \mu$. 

Moreover, we want to show that if there is a tree with $\lambda$ many branches in $V[G]$, $\lambda > \mu$, then there is such a tree in $V$ as well. 

Fix a name for a $\mu$-tree $\dot{T}$ with $\lambda$ many branches. Let us assume that either $\lambda = 0$ (so $\dot{T}$ is $\mu$-Aronszajn tree) or $\lambda \geq \mu$. For simplicity, let us assume that $\dot{T}$ is a binary tree. 

Let us show that there is a way to construct from the name $\dot{T}$ a tree in the ground model, $\tilde{T}$, which we call the \emph{termspace tree}. We will show that $\tilde{T}$ is a $\mu$-tree and bound its set of branches. %We remark that the termspace tree might have more branches than the generic one.  %Moreover, since $\mu \in \mathfrak{S}_\mu$, we only deal with trees with at least $\mu^+$ many branches.

%Let $\dot{T}$ be a name for a $\mu$-tree with $\lambda$ many branches.%, where $\lambda > \mu$. 
Let $\tilde{T}$ be the collection of all names $\dot{\eta}$ for elements of $\dot{T}$, such that the height of $\dot{\eta}$ is decided by the weakest condition of the forcing. We identify two names if they are forced to be equal and we order them by $\dot{\eta} \leq_{\tilde{T}} \dot{\eta}'$ if and only if $\Vdash \dot{\eta} \leq_{\dot{T}} \dot{\eta}'$.

Let $\dot{\eta} \in \tilde{T}$. Then $\Vdash \len (\dot\eta) = \check \alpha$, for some $\alpha < \mu$.  For each $\beta < \alpha$, $\Vdash \dot{\eta}(\check\beta) \in \check 2$.  Let $A^{\dot{\eta}}_{\beta, \epsilon} = \{q \mid q \Vdash \dot\eta(\check\beta) = \check\epsilon\}$, for $\epsilon \in 2$. Finally, let $g_{\dot\eta} \colon \alpha \to \mathcal{P}(\Add(\omega,1))^2$ be the function sending each $\beta < \alpha$ to the corresponding pair $(A^{\dot\eta}_{\beta,0}, A^{\dot\eta}_{\beta,1})$. 

\begin{claim}
For every $\dot{\eta}, \dot{\eta}' \in \tilde{T}$, $\dot\eta \leq_{\tilde{T}} \dot\eta'$ iff $g_{\dot{\eta}} \subseteq g_{\dot{\eta}'}$.
\end{claim}
\begin{proof}
Indeed, if $\Vdash \dot\eta \leq_T \dot\eta'$ then for every $\beta < \alpha = \len \dot\eta$, $\Vdash \dot{\eta}(\check\beta) = \dot{\eta'}(\check\beta)$ and thus $A^{\dot\eta}_{\beta,\epsilon} = A^{\dot\eta'}_{\beta,\epsilon}$ for all $\beta < \alpha$ and $\epsilon\in 2$.

On the other hand, if $g_{\dot\eta}$ is an initial segment of $g_{\dot\eta'}$ then clearly for every $\beta < \alpha$, there is no condition that forces $\dot\eta(\beta) \neq \dot\eta'(\beta)$.
\end{proof}
So, we can identify $\tilde{T}$ with the tree of all functions $g_{\dot\eta}$ ordered by inclusion.  

Recall that $\dot{T}$ is forced to have $\lambda$ many branches % where $\lambda \geq \mu^+$ 
and let $\{\dot{b}_\alpha \mid \alpha < \lambda\}$ be a sequence of names which is forced to be an enumeration of the branches through $\dot{T}$. By the definition of $\tilde{T}$, for each $\alpha < \lambda$ and $\zeta < \mu$, $\dot{b}_\alpha \restriction \zeta$ is a member of $\tilde{T}$.  By the definition of the tree order, $\langle \dot{b}_\alpha \restriction \zeta \mid \zeta < \mu\rangle$ is a cofinal branch. 

We need to show that $\tilde{T}$ is a $\mu$-tree and that it does not have more than $\lambda$ many different branches. Both proofs are similar, and in order to avoid repetitions let us state the following general lemma. 
\begin{lemma}\label{lemma:erdos-rado}
Let $\rho$ be a cardinal. If $\{\dot{X}_i \mid i < \rho\}$ is a sequence of names such that for all $i \neq j$, $\not\Vdash \dot{X}_i = \dot{X}_j$ and $\Vdash |\{\dot{X}_i \mid i < \rho\}| = \check\zeta$, then $\zeta \leq \rho \leq \zeta^{\aleph_0}$.  
\end{lemma}
\begin{proof}
If $\zeta = 0$ then $\rho = 0$. If $\zeta = 1$ then $\rho = 1$, otherwise, there is a condition that forces that $\dot{X}_0 \neq \dot{X}_1$.

Since the function mapping $\dot{X}_i$ to $\dot{X}_i^G$ is a surjection, $\zeta \leq \rho$. Let us assume, towards a contradiction that $\rho > \zeta^{\aleph_0}$.

Let $\{\dot{Y}_j \mid j < \zeta\}$ be a name for an enumeration of $\{\dot{X}_i \mid i < \rho\}$. By the chain condition of the forcing, for every $i$ there is a countable set $B_i \subseteq \zeta$ such that
\[\Vdash \exists j \in \check{B}_i,\, \dot{X}_i = \dot{Y}_j.\] 

By our assumption, $\zeta^{\aleph_0} < \rho$ so there is a set $A$ of size $\rho$ and $B_\star$ such that $\forall i \in A,\, B_i = B_\star$. 

For every $i < i'$ in $A$ there is a condition $p \in \Add(\omega,1)$ such that $p \Vdash \dot{X}_i = \dot{Y}_j$ and $\dot{X}_{i'} = \dot{Y}_{j'}$ for $j \neq j'$. Let us define a function $c \colon [A]^2 \to \Add(\omega,1)$ by sending the pair $\{i,i'\}$ to such a condition $p$. Since $\rho > \zeta^{\aleph_0} \geq 2^{\aleph_0}$, we can apply the Erd\H{o}s-Rado Theorem, \cite{Erdos42}, and get a homogeneous set $H$, $|H| = \aleph_1$. Let $p$ be the common color of all pairs in $[H]^2$.

For every $i \in H$, let $\zeta_i \in B_\star$ be the unique ordinal such that $p \Vdash \dot{X}_i = \dot{Y}_{\zeta_i}$. By the choice of $c$, for all $i \neq i'$, $\zeta_i \neq \zeta_{i'}$ so we obtained an injection from $\aleph_1$ to $B_\star$ which is countable---a contradiction. 
\end{proof}

Let us assume that there are $\rho$ many different elements $\dot{\eta}_i$ in the $\zeta$-th level of the tree $\tilde{T}$. The name $\dot{T}$ is forced to be a $\mu$-tree, so we can enumerate the elements in the $\zeta$-th level of $\dot{T}$ by $\langle \dot{\rho}_i \mid i < i_\star\rangle$, $i_\star < \mu$. By Lemma \ref{lemma:erdos-rado}, $\rho \leq i_\star^{\aleph_0} < \mu$. 

Similarly, let us show that $\tilde{T}$ does not have more than $\lambda$ many branches. Let $\tilde{\lambda}$ be the number of branches of $\tilde{T}$. Every name for a cofinal branch in $\tilde{T}$ provides a branch in $\dot{T}$, so by Lemma \ref{lemma:erdos-rado}, $\lambda \leq \tilde{\lambda} \leq \lambda^{\aleph_0} = \lambda$, where the last equality follows from the assumptions of the proposition.

Furthermore, if $\tilde{\lambda} > 0$ then $\lambda > 0$, as wanted. 
\end{proof}
The proof shows that if the tree property holds at $\mu$ then it holds after adding a single Cohen real, assuming that $\mu$ is $\sigma$-closed. The same statement for the non-$\sigma$-closed case (e.g.\ $\mu = \aleph_2 = 2^{\aleph_0}$ or $\mu = \aleph_{\omega+1}$) is open.

\section{Abstract Construction}\label{section: abstract construction}
The following lemma is a reformulation of a simple case of Solovay's argument for the consistency strength of the Kurepa Hypothesis, \cite[Section 4]{Je71}. A similar construction for $\kappa = \omega_1$ and $V = L[A]$, $A \subseteq \omega_1$ appears in \cite{PoSh20}. This latter construction generalizes to any successor of a regular cardinal. 
\begin{lemma}\label{lem:sealedtreeL}
Let $\kappa$ be a regular, uncountable cardinal and let us assume that $\left(\kappa^{+}\right)^{L} = \kappa^{+}$. Then there is a strongly sealed tree of height $\kappa$.
\end{lemma} 
\begin{proof}
Let us define the tree $T$. An element of $T$ is a pair $\langle M, \bar{x}\rangle$ where $M \cong \Hull^{L_{\kappa^{+}}}(\rho \cup \{x\})$, $\rho < \kappa$, $x \subseteq \kappa$ in $L$, $M$ is transitive and $\bar{x}$ is the image of $x$ under the transitive collapse. We order the tree by $\langle M, \bar{x}\rangle \leq_T \langle M',\bar{x}'\rangle$ if $M$ is the transitive collapse of $\Hull^{M'}(\rho \cup \{\bar{x}'\})$ for some ordinal $\rho$ and $\bar{x}$ is the image of $\bar{x}'$ under the transitive collapse.

If $\kappa$ is a successor cardinal, say $\kappa = \lambda^+$, then we can rearrange the tree so that all elements $\langle M, \bar{x}\rangle$ in the level $\rho > \lambda$ correspond to models in which $M \models |\rho| > \lambda$, or trivial extensions of previous nodes. In this case, $M$ is clearly an initial segment of $L_{\delta_\rho}$, where $\delta_\rho$ is some ordinal such that $L_{\delta_\rho} \models |\rho| = \lambda$. 

Let $b = \{\langle M_\alpha, x_\alpha \rangle \mid \alpha < \kappa\}$ be a branch through $T$ in some set forcing extension of $L$ which preserves $\cf(\kappa) > \omega$. Let $R_b$ be the direct limit. Then $R_b$ is well founded as it is a limit of uncountable cofinality of well founded models. Moreover, $R_b \models \text{``}V = L_{\delta}\text{''}$ for some $\delta$ and therefore $R_b \cong L_\delta$, for some $\delta < \kappa^{+}$. Let $\tilde{x}$ be the limit of the values of $x_\alpha$ along the branch and $\tilde{\kappa}$ be the limit of the values of $\kappa_\alpha$ along the branch, where $\kappa_\alpha$ is the largest cardinal in $M_\alpha$. Note that $\tilde{\kappa} \geq \kappa$ and it is the largest cardinal in $R_b$. Since $\tilde{x} \in R_b$, it is constructible. In particular, the branch $b$ is in $L$, since it is definable by the transitive collapses of the models in
\[\{\Hull^{L_{\delta}}(\rho \cup \{\tilde{x}\}) \mid \rho < \kappa\}.\]   
\end{proof}
The sealing property of the tree originated from the fact that every branch in the tree produces its own witness for its constructibility. We would like to imitate this property in the context of models of the form $J_{\alpha}[E]$, which are compatible with larger cardinals. In the proof of Lemma \ref{lem:sealedtreeL} and below, ``$\Hull$'' always denotes a fully elementary, i.e., $\Sigma_\omega$, hull. 
 
\begin{definition}
  Let $M = J_{\alpha}[E]$ be a premouse with $M \cap \Ord > \kappa$. As usual, we write $M \| \gamma$ for $J_{\gamma}[E \restriction \gamma]$. Let us assume that $\kappa$ is the largest cardinal in $M$.

We say a model $\bar M$ is an \emph{active node} if 
\begin{itemize}
\item there are ordinals $\rho, \gamma$ and $x \in \left({}^\kappa 2\right)^M$ such that $\rho < \kappa < \gamma < M \cap \Ord$ and \[ \bar M = \trcl \Hull^{M \| \gamma}(\rho \cup \{x\}), \] 
\item if $\bar{x}$ denotes the image of $x$ under the transitive collapse, then $\len \bar{x} = \rho$, and
  % This was to avoid "bad branches", so maybe with the fine structural version of the tree this is no longer necessary. But it might be used when combining the the active nodes with the sequences $s$
\item $\kappa = \rho_1(M \| \gamma)$ and $\kappa$ is the largest cardinal in $M \| \gamma$. %$\rho = \rho_1(\bar M)$, $\rho$ is the largest cardinal in $\bar M$, and
%\item if $\pi \colon \bar M \rightarrow M \| \gamma$ is the uncollapse map, then $\crit(\pi) = \rho$.
%\item for every $\rho' < \rho$, $\Hull^{\bar{M}}(\rho' \cup \{\bar{x}\}) \neq \bar{M}$, and 
%\item there is no $\rho < \bar\gamma < \bar{M} \cap \Ord$ such that $\bar{M}\restriction \bar{\gamma} = \Hull^{\bar{M} \restriction \bar{\gamma}}(\rho \cup \{\bar x\})$ and $\bar{M} \restriction \bar{\gamma} \models \ZFC^{-}$ and $\rho$ is the last cardinal.
\end{itemize}

Recall that in all interesting cases the set of $\gamma < M \cap \Ord$ such that $\rho_1(M \| \gamma) = \kappa$ is unbounded. First, note that $\rho_1(M \| \gamma)$ cannot be strictly below $\kappa$ by acceptability of the extender sequence of $M$ as $\kappa$ is a cardinal in $V$. The following observation yields that $\rho_1(M \| \gamma) \leq \kappa$ for unboundedly many $\gamma$: Let $\delta < M \cap \Ord$ be an ordinal such that $M \| \delta = \Hull^{M \| \delta}(\kappa \cup \{p\})$ for some parameter $p$. Then, in any model that contains $M \| \delta$ as an element, the hull function defines a bijection between $\kappa$ and $\delta$ that is $\Delta_1$ in the parameters $M \| \delta$ and $p$. It is easy to see that there is a bijection between $\delta$ and $\delta+\omega$ that is $\Delta_1$ definable over $M \| (\delta+\omega)$. Therefore, $\rho_1(M \| (\delta+\omega)) = \kappa$, as desired.

\begin{lemma}
Let $\bar{M}$ be an active node and let $\pi \colon \bar{M} \to M \| \gamma$ be the uncollapse map, say $\bar M = \trcl \Hull^{M \| \gamma}(\rho \cup \{x\}).$ Then:
\begin{enumerate}
\item $\rho_1(\bar{M}) = \rho$,
\item $\rho = \crit \pi$, and
\item $\pi(\rho) = \kappa$. 
\end{enumerate}
\end{lemma}
\begin{proof}
Since $\rho \subseteq \Hull^{M \| \gamma}(\rho \cup \{x\})$, we have $\crit \pi \geq \rho$. On the other hand, since $\pi(\bar{x}) = x$ and $\len \bar{x} = \rho$, $\rho \geq \crit \pi$. So $\rho = \crit \pi$ and $\pi(\rho) = \kappa$. Finally, $\rho_1(\bar{M}) = \rho$ easily follows from the full elementarity of $\pi$ as the statement ``$\kappa = \rho_1(M \| \gamma)$'' is definable.
%Let us consider now $\rho_1(\bar{M})$. The statement $\kappa = \rho_1(M \| \gamma)$, is at most $\Sigma_3$: there is a parameter $p$ such that no set $x \subseteq \kappa$ is $\{\varphi(x, p) \mid x < \rho\}$, where $\varphi$ is $\Sigma_1$. Thus by the full elementarity of $\pi$, $\pi(\rho_1(\bar{M})) = \kappa$.  
\end{proof}

\begin{lemma}\label{lemma:soundness}
Every active node is 1-sound.
\end{lemma}
\begin{proof}
Let $\bar{M}$ be an active node, say $\bar{M} = \trcl \Hull^{M \| \gamma}(\rho \cup \{x\})$ and let $\pi$ be the uncollapse map. 
Let $p$ be a good parameter for $\bar{M}$. Then $\pi(p)$ is a good parameter for $M \| \gamma$, by elementarity, and thus a very good parameter (by the 1-soundness of $M \| \gamma$). Therefore, there is a $\Sigma_1$-formula $\varphi$ with parameter $\pi(p)$ defining a surjection from $\kappa$ onto $M \| \gamma$. This is a $\Pi_2$ statement so by elementarity the same formula defines a surjection from $\rho$ onto $\bar{M}$ using the parameter $p$, as desired.  
\end{proof}

%Using Steel's condensation lemma, \cite{St10}, if $M$ is sufficiently iterable, $\bar{M}$ is either an initial segment of $M \| \kappa$ or an initial segment of an ultrapower of it by a single extender. This would be the crux of the proof of Theorem \ref{theorem: sealed trees in K}. 

Let us now define the main combinatorial object of this article, the tree $\mathbb{T}(M)$. We say a pair $\langle \bar{M}, \bar{x}\rangle$ is an \emph{active pair} if $\bar{M}$ is an active node and $\bar{x} \in {}^{{<}\kappa}2$ with $\len \bar{x} = \rho_1(\bar{M})$.
The elements of $\mathbb{T}(M)$ are triples of the form $\langle \bar{M}, \bar{x}, s\rangle$ where $\langle \bar{M}, \bar{x}\rangle$ is either an active pair or the pair $\langle \emptyset, \emptyset\rangle$ and $s \in {}^{<\kappa}2$ is such that $s^{-1}(\{1\})$ is finite and $\len s\geq \len \bar x$. 

The order of $\mathbb{T}(M)$ is defined by $\langle M_0, x_0, s_0\rangle \leq_{\mathbb{T}(M)} \langle M_1, x_1, s_1\rangle$ if and only if: 
\begin{enumerate}
\item $M_0$ is the transitive collapse of $\Hull^{M_1}(\rho_1(M_0) \cup \{x_1\})$ and $x_0$ is the image of $x_1$ under the transitive collapse, or $M_0 = x_0 = \emptyset$,
\item $s_0$ is an initial segment of $s_1$,
\item if $\langle M_0, x_0\rangle$ is an active pair, there is no $\rho'$ between $\rho_1(M_0)$ and $\len s_0$ such that $\langle \Hull^{M_1}(\rho' \cup \{x_1\}), x_1\rangle$ collapses to an active pair which is not $\langle M_0, x_0\rangle$, and
\item if $\langle M_0, x_0\rangle = \langle \emptyset, \emptyset \rangle$, there is no $\rho' < \len s_0$ such that $\langle \Hull^{M_1}(\rho' \cup \{x_1\}), x_1\rangle$ collapses to an active pair.
\end{enumerate}
\end{definition} 

The tree $\mathbb{T}(M)$ is derived from the product of two trees. The tree of the active pairs ordered by the relation which is derived from the hulls, and the tree of almost zero sequences. 
While the tree of active pairs contributes the structural complexity, the tree of the almost zero sequences adds the desirable properties of normality and splitting for free. %Without the third condition in the definition of the order of $\mathbb{T}(M)$, we would have taken a product of the trees as partial orders, which is not a tree.

\begin{lemma}\label{lem:atleast2kappabranches}
Let $\kappa$ be a cardinal in $V$ of uncountable cofinality.
Let $M$ be a premouse such that $\Hull^{M}(\kappa \cup \{p\}) \neq M$ for any $p\in M$. Then $\mathbb T(M)$ is a tree of height $\kappa$ with at least $(2^{\kappa})^M$ many branches.
\end{lemma}
\begin{proof}
First, let us verify that $\mathbb{T}(M)$ is a tree. We start by verifying that the order $\leq_{\mathbb{T}(M)}$ is transitive: Let
\[\langle M_0, x_0, s_0\rangle \leq_{\mathbb{T}(M)} \langle M_1, x_1, s_1\rangle \leq_{\mathbb{T}(M)} \langle M_2, x_2, s_2\rangle,\]
as witnessed by the ordinals $\rho_0, \rho_1$ respectively. Then we claim that the transitive collapse of $\Hull^{M_2}(\rho_0 \cup \{x_2\})$ is $M_0$ and $x_0$ is the image of $x_2$ under this collapse map. This is true since the Skolem hull of a Skolem hull is a Skolem hull. By the same reason, there is no active pair $\langle M', x'\rangle \neq \langle M_0, x_0\rangle$ below $\langle M_2, x_2\rangle$ with $\rho_1(M') < \len s_0$, as otherwise, this active pair would also be below $\langle M_1, x_1 \rangle$.

Let us assume now that $\langle M_0, x_0, s_0\rangle$ and $\langle M_1, x_1, s_1\rangle$ are both below $\langle N, y, s\rangle$ in the tree. We claim that they are compatible. We have, \[M_0 \cong \Hull^{N}(\rho_1(M_0) \cup \{y\})\text{ and }M_1 \cong  \Hull^{N}(\rho_1(M_1) \cup \{y\}).\] 
Assume without loss of generality that $\rho_1(M_0) \leq \rho_1(M_1)$. Then \[M_0 \cong \Hull^{M_1}(\rho_1(M_0) \cup \{x_1\}).\] Since $\langle M_1, x_1\rangle$ is an active pair below $\langle N, y\rangle$, $\len s_0 < \rho_1(M_1)$ and thus $s_0$ is an initial segment of $s_1$.

By the same argument, we can verify that below every element in the tree, the branch is well ordered and of length $< \kappa$. For each node $\langle \bar M, \bar x, s\rangle$ in the tree there are at least two incompatible extensions: $\langle \bar M, \bar x, s ^\smallfrown \langle 0 \rangle\rangle$ and $\langle \bar M, \bar x, s ^\smallfrown \langle 1 \rangle\rangle$. Similarly, each node in the tree has extensions of any larger height:
Consider $x\in {}^\kappa 2 \cap M$, and let $\gamma$ be the least ordinal $\leq M \cap \Ord$ such that $x \in M \| \gamma$, $M \| \gamma = \Hull^{M \| \gamma}(\kappa \cup \{x\})$, and $\rho_1(M \| \gamma) = \kappa$. In order to show that there is such an ordinal $\gamma$, let us consider the uncollapsed hull $N = \Hull^M(\kappa \cup \{x\})$. By our hypothesis on $M$, $N \neq M$. We claim that $N = M \| \delta$ for some limit ordinal $\delta$. This holds as $\kappa$ is the maximal cardinal in $M$, and therefore for every $\beta \in M$, there is a definable surjection in $M$ from $\kappa$ onto $\beta \times \omega$ and thus also onto $M \| \beta$. So $\beta \in N$ implies $M \| \beta \in N$. In particular, $N$ is an initial segment of $M$, i.e., $N = M \| \delta$ for some limit ordinal $\delta$. But $N = \Hull^M(\kappa \cup \{x\})$ and thus $\Hull^N(\kappa \cup \{x\}) = N$. Clearly, $\delta + \omega \in M$, so $\gamma = \delta + \omega$ is as desired.

The branch $b_x$ is defined by the transitive collapses of $\Hull^{M \| \gamma}(\rho \cup \{x\})$ for $\rho < \kappa$, together with the images of $x$ under these collapses and the constant sequence of zeros. 
For club many $\rho < \kappa$, the length of $\bar{x}$ is exactly $\rho$. 
%Moreover, $M \restriction \gamma$ satisfies the statement:
%\[\forall \beta,\, M \restriction \beta \neq \Hull^{M \restriction \beta}(\len x \cup \{x\}).\]
%
%Thus, by elementarity, there is no initial segment $\bar{M} \restriction \bar{\beta}$ such that \[\Hull^{\bar{M} \restriction \bar{\beta}}(\rho \cup \{\bar{x}\}) = \bar{M} \restriction \bar{\beta}.\] 
Therefore, the obtained pair is an active pair. If $x, y$ are distinct elements of ${}^\kappa 2$ and $x(\alpha) \neq y(\alpha)$, then there are no common active nodes in $b_x \cap b_y$ above $\alpha$.
\end{proof}

During the proofs of Theorems \ref{theorem: sealed trees in K} and \ref{theorem: psp and AD_R}, we will construct trees of the form $\mathbb{T}(M)$ and we will claim that they are sealed. Moreover, we will characterize all branches in terms of $x \in M$, $s \in {}^{<\kappa}2$, and certain ordinals of $M$. In order to do this we will use iterability and maximality of $M$. The next lemma will help us to exploit these properties.
\begin{lemma}\label{lem:directlimitofbranch}
Let $M$ be a countably iterable\footnote{That means, all countable substructures of $M$ are $(\omega_1, \omega_1+1)$-iterable.} premouse and let us assume that $\kappa$ is a cardinal in $V$ of uncountable cofinality. Let $b$ be a cofinal branch in the tree $\mathbb{T}(M)$ with unboundedly many active nodes and let $R_b$ be the direct limit of the active nodes on the branch, with $x_b$ being the limit of the values of $x$ in the active pairs along the branch.
Then $R_b$ is a well founded, countably iterable, sound premouse. Moreover, 
\[\rho_1(R_b) = \rho_\omega(R_b) = \len x_b = \kappa,\]
and $R_b \| \kappa = M \| \kappa$.
\end{lemma}  
\begin{proof}
    For $\alpha < \kappa$, write $\langle M_\alpha, x_\alpha, s_\alpha\rangle$ for the
    element of $b$ on level $\alpha$ of the tree. 
    By our assumption on $b$, the sequence $\langle M_\alpha \mid \alpha < \kappa \rangle$ is not eventually constant. 
    Let $I \subseteq \kappa$ be the set of all ordinals $\alpha$ 
    such that $\forall \beta < \alpha,\, M_\alpha \neq M_\beta$. 
    
    \begin{claim}
      Let $\alpha$ be a limit point of $I$. Then $\alpha \in I$. 
    \end{claim}
     \begin{proof}
    For every $\beta < \gamma$ in $I$, the inverse of the 
    collapse map $\pi_{\beta, \gamma} \colon M_\beta \to M_\gamma$ 
	is an elementary embedding, and those embeddings commute.
	Thus, the system $\langle M_\beta, \pi_{\beta, \gamma} \mid \beta, \gamma \in I \cap \alpha, \beta \leq \gamma\rangle$ is directed. Let $\tilde{M}$ be
	the direct limit of this system, and let $\tilde{x}$ be the corresponding
	limit of $x_\beta$, $\beta \in I \cap \alpha$.
	
	Let $\alpha' \in I \setminus \alpha$. Let us verify that 
	\[\tilde{M} \cong \Hull^{M_{\alpha'}}(\alpha \cup \{x_{\alpha'}\}).\]
	For every $z \in \tilde{M}$, there is $\beta \in I \cap \alpha$ and
	$\bar{z} \in M_\beta$ such that $z = \pi_{\beta, \alpha}(\bar{z})$. 
	Let $\iota \colon \tilde{M} \to M_{\alpha'}$ be defined by $\iota(z) = \pi_{\beta, \alpha'}(\bar{z})$. 
	It is routine to verify that $\iota(z)$ is independent of the choice
	of $\beta$ and that $\iota$ is an elementary map. Moreover, since  
	$\langle M_{\beta}, x_\beta\rangle$ is an active pair, $\bar{z}$ is 
	definable in $M_{\beta}$ using ordinal parameters below $\beta$ and 
	the parameter $x_\beta$. By elementarity of $\pi_{\beta, \alpha'}$, 
	$\iota(z) = \pi_{\beta, \alpha'}(\bar{z})$ is definable in $M_{\alpha'}$ 
	using ordinal parameters below $\beta$ and $x_{\alpha'}$ (since $\crit \pi_{\beta, \alpha'} \geq \beta$). 
	We conclude that the image of $\iota$ is
	$\Hull^{M_{\alpha'}}(\alpha \cup \{x_{\alpha'}\})$, as desired.
	
	In particular, $\tilde{M}$ is well founded, and it makes sense to 
	compute the length of $\tilde{x}$. Clearly, $\len \tilde{x} \geq \alpha$.
	For the other direction, if $\len \tilde{x} > \alpha$, then there
	is some ordinal $\eta \in M_\beta$, for $\beta < \alpha$, such that 
	$\pi_{\beta, \alpha}(\eta) = \alpha$ and $\eta < \len x_\beta = \beta$.
	But $\iota \circ \pi_{\beta, \alpha} = \pi_{\beta, \alpha'}$, and thus
	$\crit \pi_{\beta, \alpha} \geq \crit (\iota \circ \pi_{\beta, \alpha}) \geq \beta$,
	which contradicts our choice of $\eta$. So $\tilde{M} = \trcl \Hull^{M_{\alpha'}}(\alpha \cup \{x_{\alpha'}\})$ and $\tilde{x}$ is the image of $x_{\alpha'}$ under the collapse.
	
	Finally, there is a $\gamma < M \cap \Ord$ and some $y \in M$ such that $M_{\alpha'} \cong \Hull^{M \| \gamma}(\alpha' \cup \{y\})$ since $M_{\alpha'}$ is an active node. 
	Then, as the Skolem hull of a Skolem hull is a Skolem hull,
	\[\tilde{M} = \trcl \Hull^{M \| \gamma}(\alpha \cup \{y\})\] and $\tilde{x}$ is the image of $y$ under the transitive collapse. 
	This finishes the verification that $\langle \tilde{M}, \tilde{x}\rangle$ is an 
	active pair, and therefore $\alpha \in I$, $\tilde{M} = M_\alpha$ and $\tilde{x} = x_\alpha$.
     \end{proof}

	Let us consider $R_b$. This model is well founded as a limit of uncountable 
	length of well founded models, using the directed system 
	\[\langle M_\alpha, \pi_{\alpha, \beta} \mid \alpha, \beta \in I,\, \alpha \leq \beta\rangle.\] 
	Let $\pi_{\alpha, \kappa}$ denote the corresponding limit maps. 
	The length of $x_b$ is $\kappa$, by the same argument as above. 
	For every $z \in R_b$, there is $\beta \in I$ and $\bar{z}
	\in M_\beta$ such that $z = \pi_{\beta, \kappa}(\bar{z})$. In particular,
	$z \in \Hull^{R_b}(\beta \cup \{x_b\})$. 
	Thus, $R_b = \Hull^{R_b}(\kappa \cup \{x_b\})$. 
	
	\begin{claim}
	  $\rho_1(R_b) = \kappa$ and $R_b$ is 1-sound. In fact, $\rho_\omega(R_b) = \kappa$ and $R_b$ is sound.
	\end{claim}
	\begin{proof}
	  Pick any $\alpha \in I$. 
	Since $M_{\alpha}$ is an active node, 
	$\alpha = \rho_1(M_{\alpha}) = \crit \pi_{\alpha, \kappa}$. 
	Let us verify that $\pi_{\alpha, \kappa}(\alpha) = \kappa$. Since $\kappa$ is a
	cardinal in $V$, by acceptability, $\rho_1(R_b) \geq \kappa$ and in fact $\rho_\omega(R_b) \geq \kappa$. For each $\alpha \in I$, since $M_{\alpha}$
	is 1-sound, there is a parameter $p_\alpha$ such that 
	$\Hull_1^{M_{\alpha}}(\alpha \cup \{p_{\alpha}\}) = M_{\alpha}$. 
	The minimal such parameter is definable. 
	For every $\beta \in I \setminus \alpha$, taking 
	$p_\beta = \pi_{\alpha,\beta}(p_\alpha)$, 
	\[M_\beta = \Hull_1^{M_\beta}(\beta \cup \{p_\beta\}).\]
	Let $p_\kappa = \pi_{\alpha, \kappa}(p_\alpha)$. Let us verify that
	$\Hull_1^{R_b}(\kappa \cup \{p_\kappa\}) = R_b$. Indeed, for every $z \in R_b$
	there is $\beta \in I$ such that $z = \pi_{\beta, \kappa}(\bar{z})$ for some $\bar{z} \in M_\beta$. 
	Without loss of generality, $\beta > \alpha$, and thus 
	$\bar{z} \in \Hull_1^{M_\beta}(\beta \cup \{p_\beta\})$. In particular, 
	$z \in \Hull^{R_b}_1(\beta \cup \{p_\kappa\})$. We conclude that
	\[R_b = \Hull_1^{R_b}(\kappa \cup \{p_\kappa\}),\]
	so $\rho_1(R_b) = \kappa$ and $R_b$ is 1-sound. As $\rho_\omega(R_b) \geq \kappa$, this easily implies $\rho_\omega(R_b) = \kappa$ and hence $R_b$ is sound.
	\end{proof}
	
	\begin{claim}
	  $R_b$ is countably iterable.
	\end{claim}
	\begin{proof}
	Let $\theta$ be a sufficiently
 	large regular cardinal and let $N \prec H(\theta)$ be a countable elementary
 	substructure that contains all relevant objects (and in particular, $I$). 
 	Then $\delta = \sup(N \cap \kappa) \in \kappa$ (since $\cf(\kappa) > \omega$). 
 	We argue that
 	$\tilde{N} = R_b \cap N$ is $(\omega_1, \omega_1+1)$-iterable. 
 	This suffices by the definition countable iterability.
	Since $I$ is closed, $\delta \in I$. So 
	$M_\delta = \trcl \Hull^M(\delta \cup \{y_\delta\})$, 
	for $y_\delta \in M$. The model $M_\delta$ is countably iterable 
	as an elementary substructure of the countably iterable premouse $M$. 
	Let $j \colon R_b \cap N \to M_\delta$ be the following embedding. 
	Let $a \in R_b \cap N$. Then, by the definition of the direct limit, there is 
	some model $X = M_\alpha$ for $\alpha \in N$ and some $\bar{a} \in X$ such that 
	$a = \pi_{\alpha,\kappa}(\bar{a})$. Let $j(a) = \pi_{\alpha, \delta}(\bar{a})$. 

	Let us verify that $j$ is well defined and fully elementary. 
	So, let $a \in R_b \cap N$ and let $\bar{a}, \bar{a}'$ be two elements 
	such that $\pi_{\alpha, \kappa}(\bar{a}) = \pi_{\beta,\kappa}(\bar{a}') = a$ 
	for some ordinals $\alpha, \beta$. 
	Then, $\alpha, \beta < \sup(N \cap \kappa) = \delta$ 
	and in particular $\pi_{\alpha,\delta}(\bar{a}) = \pi_{\beta,\delta}(\bar{a}')$. 
	Thus, $j$ is well defined. 
	
	Let us show that $j$ is elementary. So, let $a \in R_b \cap N$ and let us assume that $\varphi(a)$ holds in $\tilde{N} = R_b \cap N$. 
	Then $\varphi(a)$ holds in $R_b$ (by the Tarski-Vaught criterion, 
	$\tilde{N}$ is an elementary substructure of $R_b$). 
	Therefore, if we let $\bar{a} \in M_\alpha$ be such that $\pi_{\alpha,\kappa}(\bar{a}) = a$, $\varphi(\bar{a})$ holds in $M_\alpha$, using the elementarity of $\pi_{\alpha,\kappa}$. 
	Finally, $\varphi(j(a)) = \varphi(\pi_{\alpha,\delta}(\bar{a}))$ holds in 
	$M_\delta$. 

	We conclude that there is an elementary embedding from $\tilde{N}$ to $M_\delta$ 
	and therefore $\tilde{N}$ is $(\omega_1,\omega_1+1)$-iterable. 
	\end{proof}
	This finishes the proof of Lemma \ref{lem:directlimitofbranch}.
\end{proof}

\begin{lemma}\label{lem:branchfromx}
Let $\kappa$ be a regular uncountable cardinal in $V$. Let $M$ be a sound, countably iterable premouse with largest cardinal $\kappa$.
Let $b, b'$ be two cofinal branches of $\mathbb{T}(M)$ with unboundedly many active nodes and let $\langle R_b, x_b, s_b \rangle$ and $\langle R_{b'}, x_{b'}, s_{b'} \rangle$ be the direct limits along $b$ and $b'$. Suppose that $x_b = x_{b'}$ and $s_b = s_{b'}$. Then $R_{b}$ is an initial segment of $R_{b'}$ or vice versa. 
\end{lemma}
\begin{proof}
Let $I$ be the collection of levels with active nodes in $b$ and let $I'$ be the corresponding set for $b'$. Recall that $I, I'$ are both clubs and in particular, $D = I \cap I'$ is a club. Let $\rho \in D$ and let $\langle N_\rho, x \restriction \rho, s \restriction \rho\rangle \in b$, $\langle N'_\rho, x \restriction \rho, s \restriction \rho\rangle \in b'$, where $x = x_b = x_{b'}$ and $s = s_b = s_{b'}$. Let us write $N = N_\rho$ and $N' = N'_{\rho}$.

Let $\gamma < M \cap \Ord$ and let $\pi \colon N \to M \| \gamma$ be the uncollapse map. Recall that $M \| \gamma$ is sound, $\pi$ is fully elementary, $\crit \pi = \rho = \rho_{\omega}(N)$, and $\pi(\rho) = \kappa = \rho_\omega(M \| \gamma)$. Thus, by the Condensation Lemma, \cite[Lemma 1.4]{JSSS09} (see also \cite[Theorem 5.1]{St10}), $N$ is an initial segment of $M$. By applying the same argument for $N'$, we conclude that both $N$ and $N'$ are initial segments of the same model $M$. %Without loss of generality, $N \trianglelefteq N'$. 

Let $S_0$ be the set of all $\rho \in D$ such that $N_\rho \triangleleft N_\rho'$, let $S_1$ be the set of $\rho \in D$ such that $N_\rho' \triangleleft N_\rho$ and $S_2$ be $D \setminus (S_0 \cup S_1)$. 

Let us assume that $S_0$ is stationary. Since $N_\rho \in N'_{\rho} = \Hull^{N'_{\rho}}(\rho \cup \{x \restriction \rho\})$ for $\rho \in S_0$, we can fix a Skolem term $t$ and a finite sequence $\vec{\alpha}_\rho$ of ordinals below $\rho$ such that $N_\rho = t^{N'_{\rho}}(\vec\alpha_\rho, x \restriction \rho)$. Using Fodor's lemma, we can shrink $S_0$ to $S_0'$ and stabilize $\vec\alpha = \vec\alpha_\rho$. Thus, $t^{R_{b'}}(\vec\alpha, x)$ is an initial segment of $R_{b'}$. 
Write $N_\kappa = t^{R_{b'}}(\vec\alpha, x)$. Then, using elementarity, the model $N_\kappa$ satisfies $\Hull^{N_\kappa}(\kappa \cup \{x\}) = N_\kappa$ and thus it is exactly the direct limit of the system corresponding to the branch $b$. Therefore, $N_\kappa = R_b$. 

For the case where $S_1$ is stationary we can use a symmetric argument and if $S_2$ is stationary (or even unbounded) it easily follows that $R_b = R_{b'}$. 
\end{proof}

In case that $R_b \in M$, we can conclude that it is one of the branches which were constructed in Lemma \ref{lem:atleast2kappabranches}, up to a different choice of the $s$ part and the ordinal $\gamma$.

When $\kappa$ is an inaccessible cardinal, we are done, as $\mathbb{T}(M)$ is already a $\kappa$-tree. But for successor cardinals $\kappa$ we need to show that the size of each level is $<\kappa$. %The following lemma suffices for this.

%\begin{lemma}
%Let $\kappa$ be a successor cardinal, say $\kappa = \lambda^+$, and let $\rho < \kappa$. Let $\bar M = \trcl \Hull^{M}(\rho \cup \{x\})$, where $\rho_1(M) = \rho$. Let $\zeta < \kappa$ be such that $\rho_1(M \| \zeta) = \rho$ and $M \| \zeta \models |\rho| < \lambda$. 
%Then $\bar M \trianglelefteq M \| \zeta$ or $\bar M \trianglelefteq \Ult_0(M \| \zeta', E)$ for $\zeta' \leq \zeta$ and $E \in M \| \zeta$.
% \end{lemma}
% \begin{proof}
% By the Condensation Lemma, $\bar{M}$ is either an initial segment of $M$ or an initial segment of the $\Ult_0(M,E)$ where $\mathrm{lh}(E) = \rho$. 

% In the first case, the height of this initial segment is bounded by $\zeta$. Indeed, if $\bar{M}$ contains $M \|\zeta$ then $\rho$ would be singular there. 

% In the second case, let us note that $\Ult_0(M, E)$ agrees with $M$ up to the index of $E$. In particular, if $\rho$ is regular in $\Ult_0(M,E)$ then the index of $E$ is below $\zeta$. If the index of $E$ is higher than $\zeta$, and in particular $M \| \zeta \trianglelefteq \Ult_0(M,E)$ then we are back to the first case: $\bar{M}$ has to be an initial segment of $M \|\zeta$, as the opposite direction is impossible.
% \end{proof}

\begin{lemma}
Let $\kappa$ be a regular uncountable cardinal in $V$. If $M$ is a premouse and $M \models$ ``$\kappa$ is a successor cardinal'' then $\mathbb{T}(M)$ is a $\kappa$-tree.
\end{lemma}
\begin{proof}
 Say $M \models \kappa = \lambda^+$. We analyze the size of each level $\rho<\kappa$ of $\mathbb{T}(M)$ and focus on the interesting case $\lambda < \rho$.
 For every $\rho < \kappa$ fix some $\zeta_\rho < \kappa$ such that $M \| \zeta_\rho \models |\rho| \leq \lambda$. %Let \[ \xi_\rho = \sup(\{ \zeta_{\rho'} \mid \rho' \leq \rho \}) \].
 
 Every node at level $\rho$ of $\mathbb{T}(M)$ is of the form $\langle M', x', s \rangle$ where $M'$ is an active node with $\rho_1(M') \leq \rho$, $x' \in M'$, and a binary sequence $s$ given by an element of $\rho^{{<}\omega}$. Again, we focus on the interesting case, i.e., we suppose $\lambda < \rho_1(M')$. As $M'$ is an active node, \[ M' = \trcl \Hull^{M \| \gamma}(\rho_1(M') \cup \{x\}), \] for some $\gamma < M \cap \Ord$ and some $x \in M$ that collapses to $x'$. If $\pi \colon M' \rightarrow M \| \gamma$ denotes the uncollapse map, then $\crit \pi = \rho_1(M')$ and $\pi(\rho_1(M')) = \rho_1(M\|\gamma) = \rho_\omega(M\|\gamma) = \kappa$. By the Condensation Lemma, \cite[Lemma 1.4]{JSSS09}, $M'$ is an initial segment of $M \| \gamma$. In fact, the height of this initial segment is bounded by $\zeta_\rho < \kappa$ since if $M'$ would contain $M \| \zeta_\rho$ then $\rho \geq \rho_1(M') > \lambda$ would not be cardinals in $M'$. Therefore, there are $\leq |\zeta_\rho| < \kappa$ many possible nodes of $\mathbb{T}(M)$ at level $\rho$.
\end{proof}

%When $\kappa$ is a successor cardinal, the tree $\mathbb{T}(M)$ is not necessarily a $\kappa$-tree since it might contain $\kappa$ many elements at bounded levels. Clearly, given a sealed Kurepa tree on an inaccessible cardinal $\kappa$, one can obtain a sealed Kurepa tree on $\omega_1$ simply by collapsing it to be $\aleph_1$, but it is still unclear how to obtain the same object in some canonical inner model. For example, this question is natural in $K$:
%\begin{question}
%Is there a strongly sealed Kurepa tree on $\omega_1^K$ in $K$?
%\end{question}

%\section{Trees in K}\label{section: trees in K}

\section{Applications}\label{section: applications}

We are now ready for the proof of Theorem \ref{theorem: sealed trees in K}.
\begin{proof}[Proof of Theorem \ref{theorem: sealed trees in K}]
  Using the anti-large cardinal hypothesis that there is no proper class inner model with a Woodin cardinal, we can construct the core model $K$ as in \cite{JS13} (building on \cite{St96}). Let $\kappa \geq \aleph_2$ be a regular cardinal and let us consider the tree $\mathbb{T} = \mathbb{T}(K\|\kappa^+)$. Here we are referring to $\kappa^+$ as computed in $K$. Let $b$ be a branch through $\mathbb{T}$ in some generic extension $V[G]$. By the forcing absoluteness of $K$, $K^{V[G]} = K^{V}$ and in particular, $K^V$ is still universal in $V[G]$, as in \cite{St96}. 
  
Let $R_b$ be the direct limit of the active nodes in $b$. By Lemma \ref{lem:directlimitofbranch}, $\rho_\omega(R_b) = \kappa$, $R_b$ is sound, countably iterable, and $R_b \| \kappa = K \| \kappa$. %Thus, using (for example \cite[Lemma 3.1]{NeSt16}) $R_b$ is comparable with $K \| \gamma$ for every $\gamma < \left(\kappa^+\right)^K$ such that $\rho_\omega(K \| \gamma) = \kappa$, and thus is comparable with $K \| \left(\kappa^+\right)^K$. If $R_b$ is an initial segment of $K \| \left(\kappa^+\right)^K$, we are done. Otherwise, $K \| \left(\kappa^+\right)^K$ is an initial segment of $R_b$. We want to show that the second option is impossible.
By Schindler's definition of $K\|\mu$ for $\mu \geq \omega_2^V$ (see \cite[Lemma 3.5]{GSS06}), $R_b \lhd K$. 
Since we can recover $b$ (modulo some choice of $s$) from $R_b$ and $x_b \in R_b$ as in the proof of Lemma \ref{lem:branchfromx}, this implies that the tree $\mathbb{T}(K\|\kappa^+)$ has exactly $(2^\kappa)^K = (\kappa^+)^K$ many branches in $V[G]$. If $\kappa$ is weakly compact, the Covering Lemma \cite[Theorem 3.1]{SchSt99} implies that $(\kappa^+)^K = (\kappa^+)^V$.
\end{proof}

Let us remark that in a model of $\mathrm{PFA}$, every $\omega_1$-tree has at most $\aleph_1$ many branches and is (strongly) sealed, in the sense that it is specialized. So, in this model there are no Kurepa trees. Indeed, it is unknown how to seal a given Kurepa tree without collapsing cardinals. 

\begin{question}
Is it possible to obtain a model with a strongly sealed $\kappa$-Kurepa tree using forcing? 
\end{question}

% \section{Stacking mice}\label{section: stacking mice}
We now turn to the proof of Theorem \ref{theorem: psp and AD_R}. We first recall the definition of domestic premouse from \cite{ANS01}.

\begin{definition}\label{def:domesticmouse}
  A premouse $M$ is called \emph{domestic} if there is no initial segment $N \unlhd M$ with non-empty top extender $F^N$ such that $\crit(F^N)$ is a limit of Woodin cardinals in $N$ and $\crit(F^N)$ is a limit of strong cardinals in $N\|\crit(F^N)$.
\end{definition}

Moreover, we will use Jensen's notion of a stack of mice from \cite{JSSS09}.

\begin{definition}
  Let $\cN$ be a premouse such that $\cN \cap \Ord$ is an uncountable regular cardinal. Then, if it exists, $\mathcal{S}(\cN)$ denotes the unique premouse $\mathcal{S}$ such that $\cM \unlhd \mathcal{S}$ iff there is a sound countably iterable premouse $\cM^\sta \unrhd \cN$ with $\rho_\omega(\cM^\sta) = \cN \cap \Ord$ such that $\cM \unlhd \cM^\sta$. 
\end{definition}

In the context that $K^c$ as in \cite{ANS01} exists, there is no premouse with a superstrong extender, and $\kappa$ is a regular uncountable cardinal $\mathcal{S}(K^c \| \kappa)$ as defined above exists by \cite[Lemma 3.1]{JSSS09}.

\begin{proof}[Proof of Theorem \ref{theorem: psp and AD_R}]
  Assume that there is no non-domestic premouse and let $\kappa$ be a weakly compact cardinal. Then $K^c$ as in \cite{ANS01} exists and there is no premouse with a superstrong extender. So we can consider the tree $\mathbb{T}(\mathcal{S})$ for $\mathcal{S} = \mathcal{S}(K^c \| \kappa)$ the stack of mice on $K^c \| \kappa$. We claim that $\mathbb{T}(\mathcal{S})$ is a sealed tree with exactly $\kappa^+$ many branches.

  \begin{claim}
    In $V$, $\mathbb{T}(\mathcal{S})$ has exactly $(\kappa^+)^V$ many branches.
  \end{claim}
  \begin{proof}
    By Lemma \ref{lem:atleast2kappabranches}, $\mathbb{T}(\mathcal{S})$ has at least 
    $(\kappa^+)^{\mathcal{S}} := \mathcal{S} \cap \Ord$ many branches. 
    We argue that every branch $b$ through $\mathbb{T}(\mathcal{S})$
    in $V$ is already in $\mathcal{S}$. By our anti-large cardinal hypothesis, 
    covering holds for $\mathcal{S}$ in the sense of \cite[Lemma 5.1]{JSSS09} (due to Schindler), i.e.,
    $\mathcal{S} \cap \Ord = (\kappa^+)^V$. Therefore it follows that 
    $\mathbb{T}(\mathcal{S})$ has $(\kappa^+)^V$ many branches, as
    desired.

    Let $b$ be an arbitrary branch through the tree $\mathbb{T}(\mathcal{S})$
    and consider the direct limit $\langle R_b, x_b, s_b \rangle$ along the
    branch $b$ given by Lemma \ref{lem:directlimitofbranch}.
    Let us argue that
    $R_b \unlhd \mathcal{S}$. This suffices since we can recover $b$ (modulo some choice of $s$)
    from $R_b$ and $x_b \in R_b$ as in the proof of Lemma \ref{lem:branchfromx}. 
    $\mathcal{S}$ is countably iterable by \cite[Corollary 2.11]{JSSS09}. So Lemma
    \ref{lem:directlimitofbranch} yields that $R_b$ is a sound countably iterable
    premouse with $\rho_1(R_b) = \rho_\omega(R_b) = \kappa$. 
    Therefore $R_b \unlhd \mathcal{S}$ by definition of the stack $\mathcal{S}$. 
  \end{proof}

  \begin{claim}\label{claim: sealed tree in stack}
    $\mathbb{T}(\mathcal{S})$ is sealed.
  \end{claim}
  \begin{proof}
    Let $\bP$ be a partial order satisfying the conditions in Definition \ref{def:sealed}
    and let $G$ be $\bP$-generic over $V$. By \cite[Corollary 3.4]{NeSt16}, building on
    Jensen's results in \cite{JSSS09}, we have that 
    $\mathcal{S} = \mathcal{S}^{V[G]}((K^c \| \kappa)^V)$, the stack of mice
    as constructed in $V[G]$ on top of $(K^c\|\kappa)^V$. This, or more precisely the generic absoluteness of being in the stack, is where we use the restrictions on the partial orders $\mathbb{P}$ in the definition of being sealed.
    Now let $b$ be an arbitrary branch through $\mathbb{T}(\mathcal{S})$ in $V[G]$
    and consider the direct limit $R_b$ of the active nodes in $b$ as before. 
    Then $R_b \in V[G]$ is a premouse and $K^c \| \kappa \unlhd R_b$.
    Note that $\mathcal{S}^{V[G]}((K^c \| \kappa)^V)$ is countably iterable in 
    $V[G]$ by construction. By Lemma \ref{lem:directlimitofbranch}, $R_b$ is sound and
    countably iterable in $V[G]$, and $\rho_1(R_b) = \rho_\omega(R_b) = \kappa$. So by the definition
    of the stack, $R_b \unlhd \mathcal{S}^{V[G]}((K^c \| \kappa)^V) = \mathcal{S}$ 
    and hence $b \in V$, as desired.
  \end{proof}

\end{proof}

We would like to end this paper with asking the following question:

\begin{question}
What is the consistency strength of the Perfect Subtree Property at weakly compact cardinals?
\end{question}

Until recently, it seemed reasonable that a weakly compact cardinal with the Perfect Subtree Property yields an inner model with a pair of cardinals $\lambda < \mu$ such that $\lambda$ is ${<}\mu$-supercompact and $\mu$ is inaccessible. By combining our construction in Section \ref{section: abstract construction} with results by Neeman and Steel in \cite{NeSt16} we can obtain some justification for this belief. But our belief is seriously questioned by recent results of Larson and Sargsyan \cite{LaSa21}. They showed from a Woodin cardinal that is a limit of Woodin cardinals that the $K^c$ construction in \cite{ANS01} (used in \cite{JSSS09}) can consistently break down. 

\bibliographystyle{plain}
\bibliography{References}

\end{document}